\documentclass[11pt,a4paper]{amsart}

\usepackage[a4paper,top=2cm,bottom=2cm,left=2cm,right=2cm]{geometry}
\usepackage[english]{babel}
\usepackage{newlfont}
\usepackage{fancyhdr}
\usepackage{graphicx}
\usepackage{amsmath}
\usepackage{amsfonts}
\usepackage[all]{xypic}
\usepackage{amsthm}
\usepackage{latexsym}
\usepackage[utf8]{inputenc}
\usepackage{enumerate}
\theoremstyle{plain}
\newtheorem*{thm2}{Theorem}
\newtheorem*{thm3}{Theorem 1}
\newtheorem{thm}{Theorem}[section]

\newtheorem{cor}[thm]{Corollary}
\newtheorem{lem}[thm]{Lemma}
\newtheorem{prop}[thm]{Proposition}

\theoremstyle{definition}
\newtheorem{defn}{Definition}[section]
\newtheorem{oss}{Remark}
\theoremstyle{remark}
\author{Camilla Felisetti}
\title{A support theorem for nested Hilbert schemes of planar curves}

\newcommand{\ita}{\textit}
\newcommand{\gr}{\textbf}

\newcommand{\co}{\mathbb{C}}

\newcommand{\cohilbnest}{\mathcal{C}_{b_0}^{[m,m+1]}}         			% schema di Hilbert nested
\newcommand{\cohilb}{{\mathcal{C}_{b_0}}^{[m]}} 								   	% schema di Hilbert della curva
\begin{document}
\maketitle

\begin{abstract}
	Consider a family of integral complex locally planar curves. We show that under some assumptions on the base, the relative nested Hilbert scheme is smooth. In this case, the decomposition theorem of Beilinson, Bernstein and Deligne asserts that the pushforward of the constant sheaf on the relative nested Hilbert scheme splits as a direct sum of shifted semisimple perverse sheaves. We will show that no summand is supported in positive codimension. \\
	
	\smallskip
	\noindent \tiny{CLASSIFICATION:} 14D20, 14C30
\end{abstract}

\tableofcontents

\section{Introduction}
For the rest of this article curves are assumed to be complex, integral, complete and with locally planar singularities. For ease of the reader, we remind what locally planar singularities mean.
\begin{defn}\label{lps}
	Let $C$ be a complex curve. We say that $C$ has \ita{locally planar singularities} if for every $p\in C$ the completion $\hat{\mathcal{O}}_{C,p}$ of the local ring of $C$ at $p$ can be written as
	$$ \hat{\mathcal{O}}_{C,p}=\co[[x,y]]/(f_p)$$
	for some reduced series $f_p\in \co[[x,y]]$. 
\end{defn}
Let $C$ be a curve of arithmetic genus $p_a(C)$. 
We consider the Hilbert scheme of points $C^{[m]}$, which parametrizes length $m$ finite subschemes of $C$. More precisely the $m-$th Hilbert scheme of points of $C$ is defined as 
$$C^{[m]}:= \lbrace \text{zero dimensional closed subschemes Z}\subset C \mid  \dim(\mathcal{O}_C/\mathcal{I}_Z)=m \rbrace$$
where $\mathcal{I}_Z$ is the ideal sheaf of $Z$. 

Hilbert schemes have been introduced by Grothendieck in \cite{Gr} and are now the focus of several works in mathematics. For a general introduction to Hilbert schemes of points and their properties we refer to \cite{Ko,R}. 
In \cite{AIK} and \cite{BGS}, these varieties are proved to be complete, integral, $m$ dimensional and locally complete intersections.
Moreover there is a forgetful morphism $\rho: C^{[n]}\rightarrow C^{(n)}$ from the Hilbert scheme to the symmetric product of the curve that maps any subscheme $Z$ to his support. Such a map is an isomorphism of algebraic varieties when the curve $C$ is nonsingular, while it is birational for singular curves. \\

Families of relative Hilbert schemes of plane curves are of particular interest because of their relation to curve counting questions, such as the proof of G\"{o}ttsche conjecture (see \cite{KST}) or the study of invariants of knots which arise as links of singularities (see \cite{OS}).\\

Here we consider the so called \ita{nested} Hilbert scheme $C^{[m,m+1]}$ of length $m+1$ subschemes of $C$ in which an ideal of colength 1 is fixed. More precisely we define $C^{[m,m+1]}$ as
\begin{align*}
C^{[m,m+1]}:&=\lbrace (z',z)\mid z' \in C^{[m]}, z\in C^{[m+1]}, z'\subset z\rbrace \\
&=\lbrace (I,J) \text{ ideals of }\mathcal{O}_C\mid I\subset J \text{ and } \dim(\mathcal{O}_C/J)=m, \dim(\mathcal{O}_C/I)=m+1\rbrace.
\end{align*}
 
 Also, one can consider relative versions of $C^{[m]}$ and $C^{[m,m+1]}$(see \cite{Ko}): if $\pi:\mathcal{C}\rightarrow B$ is a proper and flat family of curves, we define the relative (nested) Hilbert scheme over $B$ as 
\begin{align*}
\pi^{[m]}:& \quad \mathcal{C}^{[m]} \rightarrow B, \quad (\mathcal{C}^{[m]})_b=(\mathcal{C}_b)^{[m]};\\
\pi^{[m,m+1]}:&\quad \mathcal{C}^{[m,m+1]} \rightarrow B, \quad(\mathcal{C}^{[m,m+1]})_b=(\mathcal{C}_b)^{[m,m+1]}.
\end{align*}
The planarity of curves ensures the existence of families in which the total space of the relative Hilbert scheme is smooth (see \cite{Sh}); ultimately this is a consequence of the smoothness of the Hilbert scheme of points on a surface. Picking one of those families, the Decomposition Theorem by Beilinson, Bernstein and Deligne \cite{BBD} applied to the map $\pi^{[m]}$ asserts that the complex $R\pi^{[m]}_*\co$ decomposes as a direct sum of shifted intersection complexes associated to local systems on constructible subsets of the base.\\
Among them we find the intersection complex whose support is the whole base $B$. More precisely, if one denotes by $\tilde{\pi}: \tilde{\mathcal{C}}\rightarrow \tilde{B}$ the restriction of the family to the smooth locus, then any fiber is a smooth curve and its Hilbert scheme coincides with the symmetric product; in particular the map $\tilde{\pi}^{[m]}$ is smooth. Hence the summand of $R\pi^{[m]}_*\mathbb{Q}[m+\dim B]$ with support equal to $B$ is $\bigoplus IC_B(R^i\tilde{\pi}^{[m]}_*\mathbb{Q})[-i]$. Migliorini and Shende  showed that this is in fact the only summand.
\begin{thm2}[\cite{MS1}, Theorem 1]\label{miglioshende}
	Let $\mathcal{C}\rightarrow B$ be a proper and flat family of integral plane curves and let $\tilde{\pi}:\tilde{\mathcal{C}}\rightarrow \tilde{B}$ be its restriction to the smooth locus. If $\mathcal{C}^{[m]}$ is smooth then
	$$ R\pi^{[m]}_*\mathbb{Q}[m+dimB]=\bigoplus IC_B(R^i\tilde{\pi}^{[m]}_*\mathbb{Q})[-i].$$
\end{thm2}
In this paper we show that one can always find a base $B$ such that the relative nested Hilbert scheme is smooth and prove an analogue of theorem by Migliorini and Shende.
\begin{thm3}\label{teo}
	Let $\mathcal{C}\rightarrow B$ be a proper and flat family of integral plane curves and let $\tilde{\pi}:\tilde{\mathcal{C}}\rightarrow \tilde{B}$ be its restriction to the smooth locus. If $\mathcal{C}^{[m,m+1]}$ is smooth then
	$$ R\pi^{[m,m+1]}_*\mathbb{Q}[m+1+dimB]=\bigoplus IC_B( R^i\tilde{\pi}^{[m,m+1]}_*\mathbb{Q})[-i].$$
\end{thm3}

Let us give an overview of the paper. In section 2 we review some basic deformation theory of plane curve singularities and use this notions in section 3, where we prove the smoothness of the relative nested Hilbert scheme. In section 4 we introduce a generalized notion of discriminants, called \ita{higher discriminants}: by a theorem of Migliorini and Shende, one can prove that all the summands of the decomposition must be supported on some irreducible components of the higher discriminants, restricting significantly the possible candidates for supports. Moreover we give a criterion to determine whether a stratum is a support. The theory of higher discriminants allows us to reduce the proof of Theorem 1 for families of nodal curves. 
In section 5 we reprove Theorem 1 for families Hilbert schemes as in \cite{MS1} and prove Theorem 1 for families of nested Hilbert schemes.

\section{Versal deformations of curves singularities}

As we will systematically employ versal deformation of curve singularities (as analytic spaces), we recall here some known results.  For further details we refer to \cite{GLS}. 
\begin{defn}Let $(X,x)$ and $(S,s)$ be germs of complex analytic spaces.
	A deformation of $(X,x)$ over $(S,s)$ consists of a flat morphism $\phi: (\mathcal{X},x)\rightarrow(S,s)$ of germs of complex analytic spaces, together with an isomorphism from $(X,x)$ to the fibre of $\phi$ over $s$, $(X,x)\rightarrow (\mathcal{X}_s,x):=(\phi^{-1}(s),x)$.\\
	$(\mathcal{X},x)$ is called the \ita{total space}, $(S,s)$ the \ita{base space} and $(X,x)\cong (\mathcal{X}_s,x)$ the \ita{special fibre} of the deformation. We denote the above deformation by 
	$$ (X,x)\xrightarrow{i}(\mathcal{X},x)\xrightarrow{\phi}(S,s)$$
	or, when we want to shorten the notation, just by $(i,\phi):(\mathcal{X},x)\rightarrow(S,s)$.
\end{defn}

Roughly speaking, a versal deformation of a germ of complex analytic spaces is a deformation that contains basically all information about any possible deformation of this germ. One of the fundamental facts of deformation theory  is that any isolated singularity $(X,x)$ has a versal deformation.\\
Less informally, we might say that a deformation $(i,\phi)$ of $(X, x)$ over $(S,s)$ is versal if any deformation of $(X,x)$ over some other base space $(T,t)$ can be induced from $(i,\phi)$ by some base change $\psi:(T,t)\rightarrow(S,s)$. Moreover, if a deformation of $(X, x)$ over some subgerm $(T', t)\subset(T,t)$  is given and induced by some base change $\psi':(T',t)\rightarrow(S,s)$, then $\psi$ can be chosen in such a way that it extends $\psi'$.
Let us give now precise definitions.
\begin{defn}
	Let $(X,x)$ and $(S,s)$ be germs of complex analytic spaces.
	\begin{enumerate}[(i)]
		\item  A deformation $(i,\phi): (X,x)\xrightarrow{i} (\mathcal{X},x)\xrightarrow{\phi}(S,s)$ is called \ita{versal} if for any other deformation $(j,\varphi):(\mathcal{Y},x)\rightarrow (T,t)$ of $(X,x)$ the following holds: for any closed embedding $k:(T',t)\rightarrow (T,t)$ of complex germs and any morphism $\psi':(T',t)\rightarrow (S,s)$ there exists a morphism $\psi:(T,t)\rightarrow (S,s)$ satisfying
		\begin{enumerate}
			\item $\psi\circ k=\psi'$, and
			\item $(j,\varphi)=(\psi^*i,\psi^*\phi)$.
			\end{enumerate}
		\item  A deformation is \ita{locally versal} if it induces versal deformations of all the singularities of $X$.
		\item A versal deformation is called \ita{miniversal} if, with the notation of (iii), the Zariski tangent map $\mathrm{d}\psi: T_tT\rightarrow T_sS$ is uniquely determined by $(i,\phi)$ and $(j,\varphi)$.
		\end{enumerate}

\end{defn}

In the following section we will often use miniversal deformations since they can be described explicitly. More precisely let $(C,0)$ be the germ at the origin of the zero locus of some $f\in \co[x,y]$ such that $f(0)=0$. Fix $g_1\ldots g_t \in \co[x,y]$ whose images form a basis of the vector space  $\co[x,y]/(f,\partial_x f,\partial_y f)$. Then consider $F:\co^t\times \co^2\rightarrow \co^t\times\co$ given by $F(u_1,...,u_t,x,y) = (u_1,\ldots,u_t, \sum_i (f + g_iu_i)(x,y))$. Taking the fibre over $\co^t\times 0$ gives a family of curves over $\co^t$; taking germs at the origin gives the miniversal deformation $(\mathcal{C},0) \rightarrow(\co^t,0)$ of $C$. Moreover, if $g_1',\ldots,g_s'\in \co[x,y]$ are any functions and $(\mathcal{C}',0)\rightarrow(\co^s,0)$ the analogously formed deformation of $C$, then the tangent map $\co^s\rightarrow \co[x,y]/(f,\partial_x f,\partial_yf)$ is just induced by the quotient $\co[x,y]\rightarrow \co[x,y]/(f,\partial_x f,\partial_yf)$. As soon as this map is surjective, the family $(\mathcal{C}',0)\rightarrow(\co^s, 0)$ is itself miniversal.\\

Now we would like to have a measure of "how singular" a curve is, for example we could look at how far a curve is from its normalization. Given a singular curve $C$ and denoted its normalization by $\overline{C}$, we define the cogenus $\delta$ to be the difference between its arithmetic and geometric genera $\delta(C):=p_a(C)-p_a(\overline{C})$. 
For example, the cogenus of a curve with one node is precisely 1. 
The following theorem, due to Teissier, shows why the cogenus is a good candidate for our purpose. Moreover it will be a key result to reduce the proof of Theorem 1 to the case of a family of nodal curves.

\begin{thm}[\cite{T}]\label{T}
	Let $\mathcal{C}\rightarrow B$ be a family of curves. Then the cogenus is an upper semicontinuous function on $B$. Local versality is an open condition and in a locally versal family the locus of $\delta$-nodal curves is dense in the locus of curves with cogenus at least $\delta$. In particular, the locus of curves of cogenus $\delta$ in a locally versal family has codimension $\delta$.
\end{thm}

As we are working with the cogenus we would like to have a result that allows us not to care about $p_a(\overline{C})$. In \cite{L} Laumon showed that any curve singularity can be found on a rational curve.  Moreover, given a family of curves $\mathcal{C}\rightarrow B$, then around any point $b_0\in B$ one can find a different family $\pi':\mathcal{C}'\rightarrow B'$ of rational curves such that $\mathcal{C}'_{b_0}$ is rational with the same singularities\footnote{By having the same singularities, we mean that the completions of local rings at two corresponding singular points (see Definition \ref{lps}) are isomorphic.} as $\mathcal{C}_{b_0}$ and the two families induce the same deformations of the singularities of the central fiber. 
This is a consequence of the following proposition.

%\begin{prop}[\cite{FGVs}]
%	The map from the base of a versal deformation of an integral locally planar curve to the product of the versal deformations of its singularities is a smooth surjection.
%\end{prop}

\begin{prop}[\cite{MS1}, Cor. 6]\label{rat}
	Let $\pi: \mathcal{C}\rightarrow B$ be a family of curves. Fix $b_0 \in B$, and let $\overline{\mathcal{C}}_{b_0}$ be the normalization of $\mathcal{C}_{b_0}$.Then there exists a neighbourhood $b_0\in B'\subseteq B$ and a family $\pi': \mathcal{C}'\rightarrow B'$ such that $\mathcal{C}'_{b_0}$ is rational with the same singularities as $\mathcal{C}_{b_0}$, and $\mathcal{C}$ and $\mathcal{C}'$ induce the same deformations of these singularities on $B'$. In particular, they have the same discriminant locus.\\ Moreover, if $\tilde{\pi}:\mathcal{C}\rightarrow \tilde{B}$ and $\tilde{\pi}':\mathcal{C}'\rightarrow \tilde{B}'$ are the restrictions of the families to the locus of points with smooth fibre, on $\tilde{B}'$ we have an equality of local systems $R^1\tilde{\pi}_*\co=R^1\tilde{\pi}_*'\co\oplus H^1(\overline{\mathcal{C}}_{b_0})$, where $H^1(\overline{\mathcal{C}}_{b_0})$ denotes the constant local system with this fiber.
\end{prop}

To make use of such a replacement we need to know that $\mathcal{C}'^{[m,m+1]}$ is smooth if $\mathcal{C}^{[m,m+1]}$ is.  This follows from results on the smoothness of the nested Hilbert scheme which we are going to show. The results and their proof are closely analogous to \cite[Prop.17 and Thm.19]{Sh}, in which they are stated for $\mathcal{C}^{[m]}$.

\section{Smoothness of the relative nested Hilbert scheme}

Let $V\subset \co[x,y]$ be a finite dimensional smooth family of polynomials and consider the family of curves
$$\mathcal{C}_V:=\{(f,p)\in V\times \co^2\mid f(p)=0\}.$$
If we consider the associated family of nested Hilbert scheme $\mathcal{C}_V^{[m,m+1]}$ then it is included in  $V\times (\co^2)^{[m,m+1]}$. In \cite{C}, Cheah shows that the nested Hilbert scheme $(\mathbb{C}^2)^{[m,m+1]}$ is nonsingular for all $m$. Moreover she gives an explicit description of its tangent space:  if $(I,J)$ is a pair of ideals of $\co[x,y]$ with $I\subseteq J$ such that $(I,J)$ defines a point in $(\mathbb{C}^2)^{[m,m+1]}$, then the tangent space $T_{(I,J)}(\co^2)^{[m,m+1]}$ is isomorphic to $Ker(\phi-\psi)$ where
$$\begin{array}{lcc}
\phi: Hom_{\co[x,y]}(I,\co[x,y]/I)\rightarrow Hom_{\co[x,y]}(I,\co[x,y]/J)&&\\
\psi: Hom_{\co[x,y]}(J,\co[x,y]/J)\rightarrow Hom_{\co[x,y]}(I,\co[x,y]/J)&&\\
\end{array}$$
are the obvious maps and $$(\phi-\psi): Hom_{\co[x,y]}(I,\co[x,y]/I)\oplus Hom_{\co[x,y]}(J,\co[x,y]/J) \rightarrow Hom_{\co[x,y]}(I,\co[x,y]/J)$$ is defined as  $(\phi-\psi)(\eta_1,\eta_2):=\phi(\eta_1)-\psi(\eta_2)$.\\
Let us detail this isomorphism a little bit. The tangent space $T_{J}(\co^2)^{[m]}$ to the Hilbert scheme $(\co^2)^{[m]}$ in an ideal $J$ is canonically isomorphic to $Hom_{\co[x,y]}(J,\co[x,y]/J)$ and the isomorphism is constructed in the following way. Given an element $\eta\in Hom_{\co[x,y]}(J,\co[x,y]/J)$ we choose a lifting $\tilde{\eta}:J\rightarrow \co[x,y]$ and such a lifting gives a tangent vector $J_{\varepsilon,\eta}=J+\epsilon\tilde{\eta}(J) $. The fact that $\eta$ is a morphism of $\co[x,y]$-modules ensures that $J_{\epsilon,\eta}$ is indeed an ideal of $\co[x,y,\varepsilon]/(\varepsilon^2)$ and thus that it defines a tangent vector.\\
Now we observe that $$T_{(I,J)}(\co^2)^{[m,m+1]}\subset T_{I}(\co^2)^{[m+1]}\oplus T_{J}(\co^2)^{[m]}\cong Hom_{\co[x,y]}(I,\co[x,y]/I)\oplus Hom_{\co[x,y]}(J,\co[x,y]/J).$$
The last isomorphism sends a pair $(\eta,\zeta)$ in a couple of tangent vectors $$(I_{\varepsilon,\eta},J_{\varepsilon,\zeta}) \qquad\text{with }I_{\varepsilon,\eta}=I+\varepsilon\tilde{\eta}(I), \quad J_{\varepsilon,\zeta}=J+\varepsilon\tilde{\zeta}(J),$$
that do not satisfy the condition $I_{\varepsilon,\eta}\subseteq J_{\varepsilon,\zeta}$ a priori;
this is ensured precisely by requiring that $(\eta,\zeta)$ lies in  $Ker(\phi-\psi)$.

Choose a polynomial $f\in I\subset J$. If we write $(\tilde{I},\tilde{J})$ for the image of the couple $(I,J)$ in $\co[x,y]/(f)$ then we have an exact sequence of vector spaces
\begin{equation}\label{es}
0\rightarrow T_{f,(\tilde{I},\tilde{J})}\mathcal{C}_V^{[m,m+1]}\rightarrow T_fV\times T_{(I,J)}(\co^2)^{[m,m+1]}\rightarrow \co[x,y]/I,
\end{equation}
where the last map is given by $$(f+\varepsilon g,(\eta,\zeta))\mapsto \eta(f)-g \text{ mod }I.$$
Even though $\zeta$ do not intervene explicitly in the last map, the condition $\eta(f)-g\equiv 0 \text{ mod }I$ ensures that infinitesimally $f+\varepsilon g$ is contained in $I_{\varepsilon,\eta}$. Since $(\eta,\zeta)\in Ker(\phi-\psi)$, $I_{\varepsilon,\eta}\subset J_{\varepsilon,\zeta}$; thus $f+\varepsilon g$ belongs to $J_{\varepsilon,\zeta}$ as well.\\
Now, we observe that if $f$ is reduced then all the fibers in a neighbourhood $U$ of $f$ are reduced and the relative nested Hilbert schemes $\mathcal{C}_U^{[m,m+1]}$ are reduced of pure dimension $\dim V+m+1$. Also they are locally complete intersections (see \cite{BGS}). Then $\mathcal{C}_V^{[m,m+1]}$ is smooth at a point $(f,(I,J))$ if the tangent space at this point has dimension $m+1+\dim V$. \\
Looking at dimensions of the vector spaces in (\ref{es}), one notices that $\dim T_fV=\dim V$ as $V$ is supposed to be smooth,  $\dim T_{(I,J)}(\co^2)^{[m,m+1]}=2m+2$ by \cite{C} and finally $\co[x,y]/I$ has dimension $m+1$ by hypothesis: as a result $\dim T_{f,(\tilde{I},\tilde{J})}\mathcal{C}_V^{[m,m+1]}=\dim V+m+1$ if and only if the last map in (\ref{es}) is surjective. 
The easiest way to ensure this is to ask for surjectivity already in the case $\eta=\zeta=0$, that is  $T_fV\rightarrow \co[x,y]/I$  is surjective. \\
We are now ready to prove the smoothness of the relative nested Hilbert scheme.
 
\begin{prop}\label{p17}
	Let $\mathcal{C}\rightarrow \mathbb{V}$ a family of versal deformations with base point $0\in \mathbb{V}$. Up to restricting $\mathbb{V}$ to a small neighbourhood of 0,  the relative nested Hilbert scheme $\mathcal{C}^{[m,m+1]}$ is smooth.
\end{prop}
\begin{proof}
%	In \cite{FGVs} it is proved that in this hypotheses the relative compactified Jacobian is smooth and so is their product. Therefore by corollary (\ref{ver}) and point $(ii)$ of theorem \ref{s1} $\mathcal{C}_{\mathbb{V}}^{[m,m+1]}$ is smooth.\\
Suppose $f$ is the polynomial defining $\mathcal{C}_0$. Choose $\mathbb{V}\subset \co[x,y]$ containing $f$ such that $\mathcal{C}_{\mathbb{V}}\rightarrow \mathbb{V}$ is a versal deformation of the singularity of $\mathcal{C}_0$ and $T_f\mathbb{V}$ contains all polynomials of degree $\leq m$. Then $T_f\mathbb{V}$ will be of dimension $\geq m+1$, thus for any $I$ of colength $m+1$, $T_f\mathbb{V}$ will project surjectively onto $\co[x,y]/I$. By the considerations above, the dimensions counting in (\ref{es}) implies that the relative nested Hilbert scheme $\mathcal{C}^{[m,m+1]}$ is smooth. 
\end{proof}
%
%From the proof we can also deduce the following corollary, which holds also for not nested Hilbert schemes; see (\cite{Sh}, Proposition 14).
%\begin{cor}
%	Let $\mathcal{C}\rightarrow \mathbb{V}$ a family of versal deformations with base point $0\in \mathbb{V}$ such that the relative nested Hilbert scheme $\mathcal{C}_{\mathbb{V}}^{[m,m+1]}$ is smooth. Then $\mathcal{C}_{\mathbb{V}}^{[h,h+1]}$ is smooth for all $h\leq m$.
%	\end{cor}

\begin{oss}
	The smoothness of the relative nested Hilbert scheme over any versal deformation is equivalent to the smoothness over the miniversal deformations. In fact, if $\overline{\mathcal{C}}\rightarrow \overline{\mathbb{V}}$ is the miniversal deformations there are compatible isomorphisms of germs $\mathbb{V}\cong \overline{\mathbb{V}}\times (\co^t,0)$  and $\mathcal{C}\cong \overline{\mathcal{C}}\times (\co^t,0)$ and hence also 
	$\mathcal{C}^{[m,m+1]}\cong \overline{\mathcal{C}}^{[m,m+1]}\times (\co^t,0)$
\end{oss}

From the smoothness of the relative nested Hilbert scheme we can deduce an analogue result as the one in \cite[Theorem 8]{MS1}. 
\begin{cor}\label{thm8}
	Let $\mathcal{C}_B\rightarrow B$ a family of curves. Given $b_0\in B$, let $(\mathbb{V},0)$ be the product of the versal deformations of singularities of $\mathcal{C}_{b_0}$. Moreover suppose that the relative Hilbert schemes $\mathcal{C}_B^{[d,d+1]}$ are smooth for all $d<m$.
	\begin{enumerate}[(i)]
		\item  The smoothness of $ \mathcal{C}_B^{[m,m+1]}$ depends only on the image $\mathcal{T}$ of $T_{b_0} B$ in $T_0\mathbb{V};$
		\item if $\mathcal{C}_B^{[m,m+1]}$ is smooth along $\mathcal{C}^{[m,m+1]}_{b_0}$ then $\dim\mathcal{T}\geq \min (\delta(\mathcal{C}_{b_0}), m+1)$;
		%\item if $\dim\mathcal{T}\geq m+1$ and $\mathcal{T}$ is general among such subspaces, then $\mathcal{C}^{[m,m+1]}$ is smooth $\mathcal{C}^{[m,m+1]}_{b_0}$;
		\item $\mathcal{C}_B^{[m,m+1]}$ is smooth along $\mathcal{C}^{[m,m+1]}_{b_0}$ if and only if $\mathcal{T}$ is transverse to the image of $T_x\mathcal{C}_{b_0}^{[m,m+1]}$ in $T_0\mathbb{V}$ for all $x\in \mathcal{C}_{b_0}^{[m,m+1]}$. It suffices for $\mathcal{T}$ to be generic of dimension at least $\delta(\mathcal{C}_{b_0})$.
	\end{enumerate}
\end{cor}
\begin{proof}
	To prove $(i)$ take a subscheme $z\in \mathcal{C}_{b_0}^{[m,m+1]}$ which decomposes as $$z=(z_0,\ldots, z_k)$$
	such that $z_0=(z_{01},z_{02})\in \mathcal{C}_{b_0}^{[d_0,d_0+1]}$ is a subscheme supported at a point $c_0$ and $z_i\in \mathcal{C}_{b_0}^{[d_i]}$ are length $d_i$ subschemes supported on points $c_i$.\\
	 Let $(\overline{\mathcal{C}}_i,c_i)\rightarrow (\mathbb{V}_i,0)$ be the miniversal deformations of the singularities $(\mathcal{C}_{b_0},c_i)$ and $(B,b_0)\rightarrow \prod(\mathbb{V}_i,0)$ a map along which $\coprod(\overline{\mathcal{C}}_i,c_i)\rightarrow (B,b_0)$ pulls back. 
	 We denote by $\mathcal{C}_{\mathbb{V}}\rightarrow \mathbb{V}:=\prod \mathbb{V}_i$ the induced family of deformations of $\mathcal{C}_{b_0}$.\\
	 Then analytically locally, the germ $(\mathcal{C}_B^{[m,m+1]},[z])$ pulls back from $(\overline{\mathcal{C}}_0^{[d_0,d_0+1]},[z_0])\cdot\prod (\overline{\mathcal{C}}_i^{d_i},[z_i])$ along the same map. We observe that the fibres of $(\overline{\mathcal{C}}_i^{d_i},[z_i])\rightarrow (\mathbb{V}_i,0)$ are reduced of dimension $d_i$ by \cite{AIK} and the total space is nonsingular by \cite[Prop. 17]{Sh}. Moreover the same holds for $(\overline{\mathcal{C}}_0^{[d_0,d_0+1]},[z_0])\rightarrow \mathbb{V}_0$ by proposition (\ref{p17}).
	  As the $\mathbb{V}_i$ were taken miniversal, the map $T_{b_0}B\rightarrow T\mathbb{V}=\prod T_0\mathbb{V}_i$ is uniquely defined and the smoothness of the pullback depends only on the image $\mathcal{T}$ of such a map. 
	  
	To check $(ii)$ we might assume by $(i)$ that the map $T_{b_0}B\rightarrow \prod T_0\mathbb{V}_i$ is an isomorphism and identify locally $B$ with its image $\overline{B}\subseteq\prod (\mathbb{V}_i,0)$. We can shrink $\overline{B}$ until it can be written as $B\times \mathbb{D}^{k}$ for some polydisc $\mathbb{D}^{k}$; as smoothness is an open condition we may shrink $\mathbb{D}^{k}$ further until $\mathcal{C}^{[m,m+1]}_{\mid B\times \epsilon}$ is smooth for all $\epsilon \in \mathbb{D}^k$.  By theorem \ref{T}, the locus of nodal curves with the same cogenus as $\mathcal{C}_{b_0}$ in $\prod \mathbb{V}_i$ is nonempty and of codimension $\delta(\mathcal{C}_{b_0})$; choose an $\epsilon$ such that $B\times \epsilon$ contains a point $p$ corresponding to such a curve. If $m+1\geq \delta$ the statement is trivial. If $m+1\leq \delta$, we can find a point $z\in \mathcal{C}_p^{[m,m+1]}$, which is a subscheme supported at $m+1$ nodes. The Zariski tangent space $T_z \mathcal{C}_p^{[m,m+1]}$ has dimension $2m+2$, therefore $\mathcal{C}_p^{[m,m+1]}$  cannot be smoothed over a base of dimension less than $m+1$. 
	%For point $(iii)$, we assume as above that $B$ is embedded in $\overline{B}=\prod \mathbb{V}_i$. As the dimension of $\mathcal{T}$ is greater equal than $m+1$, then by lemma (\ref{lem18}) it is transverse to any ideal of colength $\leq m+1$, therefore the relative nested Hilbert scheme is smooth. 
	
	Finally, $(iii)$ is stated for the relative compactified Jacobian in \cite[Corollary B.3]{FGVs} and the prove can be copied almost line by line. Since $\mathcal{C}^{[m,m+1]}_B=\mathcal{C}^{[m,m+1]}_{\mathbb{V}}\times_{\mathbb{V}} B$, by implicit function theorem the image $\mathcal{T}$ of $T_{b_0}B$ in $T_0\mathbb{V}$ is transverse to $T_x\mathcal{C}^{[m,m+1]}_{b_0}$ if and only the relative nested Hilbert scheme is smooth .
	%Finally, $(iv)$ is stated for the compactified Jacobian in \cite{FGVs}. For $m>>0$ the statement is true because $\mathcal{C}^{[m,m+1]}$ fibers smoothly over the product of the Jacobians; for lower $m$ this is corollary (\ref{ver}). 
\end{proof}
\begin{cor}\label{codim}
If $\mathcal{C}\rightarrow B$ is a family of curves with $\mathcal{C}^{[m,m+1]}$ smooth, then for $\delta\leq m+1$, the locus of curves of cogenus $\delta$ is of codimension at least $\delta$ in $B$.
\end{cor}
\begin{proof}
	Suppose not and consider a generic $\delta-1$ subvariety $B'$ of $B$. Then the restriction $\mathcal{C}^{[m,m+1]}_{B'}\cong \mathcal{C}^{[m,m+1]}_{B}\times_{B} B'$ is nonsingular and intersects the locus of curves of cogenus $\delta$, but this contradicts item $(ii)$ of theorem \ref{thm8}.
	\end{proof}
\section{Supports}

For ease of the reader let us recall the statement of the decomposition theorem for nonsingular varieties.
\begin{thm}[\gr{Decomposition theorem}]\label{dec}
	Let $f:X\rightarrow Y$ be a proper map of nonsingular complex algebraic varieties. Then there exists a finite collection of constructible sets $Y_{\alpha}$ and local systems $\mathcal{L}_{\alpha}$ on $Y_{\alpha}$ such that the local system $Rf_*\mathbb{Q}[dim X]$ decomposes in the derived category of constructible sheaves as 
	\begin{equation}\label{dec3}
	Rf_*\mathbb{Q}[dim X]\cong \bigoplus_{\alpha} IC_{\overline{Y}_{\alpha}}(\mathcal{L}_{\alpha})[\dim X-\dim Y_{\alpha}].
	\end{equation}
\end{thm}

\begin{defn}
We call \ita{supports of f} the $Y_{\alpha}$ appearing in equation (\ref{dec3}). 
\end{defn}
 We want describe the supports of the map $$\pi^{[m,m+1]}:\mathcal{C}^{[m,m+1]}\rightarrow B.$$ 
 Clearly among them we can always find the smooth locus $\tilde{B}$ of the family and the summand supported on $B$ is given by the direct sum of the cohomology sheaves $\bigoplus IC_B(R^i\tilde{\pi}^{[m]}_*\co)[-i]$, but a priori we could have other summands supported on subsets of positive codimension. \\
In general, it is not easy to determine the supports of a given map $f:X\rightarrow Y$. However, there exists a fairly general approach to the so called \ita{support type theorems} like the decomposition theorem, which was developed by Migliorini and Shende in \cite{MS2}.
 Such an approach relies on the fact that even though a stratum $S$ might be necessary in a Whitney stratification of a map $f$, the change in the cohomology of the fibres of $S$  can be predicted just by looking at the map on the strata containing $S$.\\
 Therefore, Migliorini and Shende constructed a coarser stratification, the \ita{stratification of higher discriminants}. This description refines the notion of discriminant: instead of looking at the inverse images of points one can consider the inverse images of discs $\mathbb{D}^r$ of varying dimension $r$. Clearly the bigger the disc is the more likely its inverse image will be nonsingular. Let us be more precise: suppose $Y$ is nonsingular and let $Y=\bigsqcup Y_{\alpha}$. Take $y\in Y$ and let $k$ be the dimension of the unique stratum containing $y$. Consider the codimension $k$ slice, meeting the stratum only in $y$. Its inverse image will be a nonsingular codimension $k$ subvariety of $X$. In case $Y$ is singular, we might be more careful about what we mean by "disc": we choose a local embedding $(Y,y)\subset (\co^n,0)$ and define a disc as the intersection of $Y$ with a nonsingular germ of complete intersection $T$ through $y$. The dimension of the disc is $\dim Y-\mathrm{codim} T$.
\begin{defn}Let $f:X\rightarrow Y$ be a morphism of algebraic varieties. We define the $i-th$ higher discriminant $\Delta^i(f)$ as:
	\begin{align*}
	\Delta^i(f):= \lbrace y\in Y\mid \text{there is no }(i-1)-\text{dimensional disc }\phi:\mathbb{D}^{i-1}\rightarrow Y,&\\
	\text{with } f^{-1}(\mathbb{D}^{i-1}) \text{ non singular , and }\mathrm{codim}(\mathbb{D}^{i-1},Y)=\mathrm{codim} (f^{-1}(\mathbb{D}^{i-1}),X)\rbrace.&
	\end{align*}
\end{defn}
The higher discriminants $\Delta^i(f)$ are closed algebraic subsets, and $\Delta^{i+1}(f)\subset \Delta^i(f)$ by the openness of nonsingularity and the semicontinuity of the dimension of the fibres. Also we would like to remark that 
$\Delta^1(f)$ is nothing but the discriminant $\Delta(f)$ that is the locus of $y\in Y$ such that $f^{-1}(y)$ is singular. 

One advantage of  higher discriminants is that they are usually much easier to determine via differential method than the strata of a Whitney stratification. As we are supposing $Y$ to be nonsingular, the implicit function theorem prescribes precise conditions under which the inverse image of a subvariety by a differentiable map is nonsingular: the tangent space of the subvariety must be transverse to the image of the differential. 
Hence, under this assumption we have the following 
\begin{prop}\label{hd}
	\begin{align*}
	\Delta^i(f):= \lbrace y\in Y\mid \text{ for every linear subspace } I\subset T_yY, \text{with }\dim I=i-1, &\\
	\text{the composition  }T_xX\xrightarrow{df}T_yY\rightarrow T_yY/I \text{ is not surjective for some } x\in f^{-1}(y)\rbrace&
	\end{align*}
\end{prop}
We may rephrase condition of proposition (\ref{hd}) saying that there is no $(i-1)$- dimensional subspace $I$ transverse to $f$.\\
The following result shows the relevance of the theory of higher discriminants in determining the summands appearing in the decomposition theorem. 
\begin{thm}[\cite{MS2},Theorem B]\label{thmB}
	Let $f:X\rightarrow Y$ be a map of algebraic varieties. Then the set of $i$-codimensional supports of the map $f$ is a subset of the set of $i$-codimensional irreducible components of $\Delta^i(f)$.
\end{thm}
This theorem restricts significantly the set of candidates for the supports.
Furthermore, to check whether a component of a discriminant is relevant it is enough to check its generic point. \\

\subsection{Supports of $\pi^{[m,m+1]}$}
We now want to construct a stratification of $B$ such that the strata are precisely the higher discriminants of the map $\pi^{[m,m+1]}:\mathcal{C}^{[m,m+1]}\rightarrow B$. 
Let $b_0\in B$ be the base point of $B$ and suppose $\mathcal{C}_{b_0}=C$ is the curve with the highest cogenus, which we call $\delta$.
For any $i=0\ldots \delta$ we set
$$B_i:=\{b\in B\mid \delta(\mathcal{C}_b)\geq i\}.$$
Clearly $B=\bigcup_{i}B_i$ with $B_{i+1}\subset B_{i}$. As in the case of higher discriminants, we notice that $B_0$ is the nonsingular locus of the family. By corollary \ref{codim}, the strata $B_i$ have codimension at least $i$ in $B$. 

We want to show the following proposition:
\begin{prop}
	Let $\pi:\mathcal{C}\rightarrow B$ be proper flat family of curves such that the relative nested Hilbert scheme $\pi^{[m,m+1]}:\mathcal{C}^{[m,m+1]}\rightarrow B$ is nonsingular for any $m$. Let $\delta$ be the highest cogenus we can find on a curve in the family. Then for any $i=0\ldots\delta$
	$$\Delta^i(\pi^{[m,m+1]})\subseteq B_i.$$
	\end{prop}
\begin{proof}
%	Let $b\in B_i$. As the relative nested Hilbert scheme is nonsingular at $b$, then by items $(ii)-(iii)$ of corollary (\ref{thm8}) then the image $\mathcal{T}$ of $T_bB$ into the product of the first order deformations of the singularities $C_b$ must be of dimension greater or equal than $i$. Therefore we have that $B_i\subseteq \Delta^i(\pi^{[m,m+1]})$.
	Suppose $b\in\Delta^i(\pi^{[m,m+1]})$. Then there is no $(i-1)$-dimensional subspace of $T_bB$ transversal to the image of $T\mathcal{C}_b^{[m,m+1]}$. If the cogenus of $\mathcal{C}$ were $<i$, then by item $(iii)$ of theorem $\ref{thm8}$, $\mathcal{T}$ could have dimension $<i$ and be transversal , contradicting the hypothesis. 
\end{proof}

As a consequence of theorem \ref{thmB} if we have supports different from the smooth locus, then we will have to look for them in the $i$-codimensional irreducible components of the $B_i$'s, consisting of curves whose cogenus is exactly equal to $i$.\\
We will prove Theorem 1 using the a criterion on supports coming from mixed Hodge theory: the stalks of $IC$ sheaves appearing in the decomposition theorem are endowed with a mixed Hodge structure; moreover Saito \cite{Sa} proves that the isomorphism 
$$ H^k(f^{-1}(y))=\mathcal{H}^k(Rf_*\mathbb{Q})_{y}\cong \bigoplus_{\alpha} \mathcal{H}^k(IC_{\overline{Y}_{\alpha}}(\mathcal{L}_{\alpha}))_y$$
in the decomposition theorem is actually an isomorphism of mixed Hodge structures.  
Whenever we have a mixed Hodge structure $H=\oplus H^i$ we can define the so called \ita{weight polynomial} as 
$$ \mathfrak{w}(H)(t):=\sum(-1)^{i+j}t^i \dim \mathrm{Gr}_i^{W}H^j\quad \in \mathbb{Z}[t].$$
This polynomial has the additivity property, i.e. if $Z\subset X$ is a closed algebraic subvariety of $X$ then
$$\mathfrak{w}(H^*(X))(t)=\mathfrak{w}(H^*(X\setminus Z))(t)+\mathfrak{w}(H^*(Z))(t).$$
We have the following criterion:
\begin{prop}\cite[Prop. 15]{MS1}\label{critpesi}
	Suppose $f:X\rightarrow Y$ is a proper map between nonsingular algebraic varieties. Let $\mathcal{F}$ be a summand of $Rf_*\mathbb{Q}[\dim X].$ Given $y\in Y$, we set $X_y:=f^{-1}(y)$. If for all $y\in Y$ $\mathfrak{w}(\mathcal{F}_y[-\dim X])=\mathfrak{w}(X_y)$, then $\mathcal{F}=Rf_*\mathbb{Q}[\dim X]$.
\end{prop}

%In our case, we choose a point $b_0\in B$, set $\mathcal{F}_{b_0}=\left(\bigoplus IC_B(R^i\tilde{\pi}^{[m]}_*\mathbb{Q})[-i]\right)_{b_0}$ and we prove that the weight polynomial associated to it is equal to the weight polynomial of $(R\pi^{[m,m+1]}_*\mathbb{Q}[\dim X])_{b_0}$.
First we show the result for the Hilbert scheme in (\cite{MS1}) with a direct computation, then we proceed to prove our theorem for the nested case. As we remarked above, the criterion can be verified just on the generic points of the strata. By theorem \ref{T} the generic points of the $B_i$ are nodal curves. Therefore one can reduce the proof of Theorem 1 to the case of a family of nodal curves. 
Using proposition \ref{rat} and the techniques in \cite{MS1}, one can suppose that all the curves are rational. As a result the arithmetic genus of the curves will coincide with their cogenus. \\

%=============================
% PROOF OF THEOREM
%=============================

\section{Proof of theorem 1}

Let $\pi:\mathcal{C}\rightarrow B$ a proper flat family of rational nodal curves locally versal at a base point  $b_0\in B$. Call $\delta:=\delta(\mathcal{C}_{b_0})$. Consider the nodes $\{x_1,\ldots, x_{\delta}\}$ of the central fiber $\mathcal{C}_{b_0}$. Shrinking $B$ if necessary, we can assume the following facts:
\begin{enumerate}[1)]
	\item The discriminant locus is normal crossing divisor $\Delta:=\bigcup D_i$ with $i=0,\ldots,\delta$, where $D_i$ is the locus in which the $i-$th node $x_i$ is preserved. 
	\item If $b\in B$ is such that $\mathcal{C}_b$ is nonsingular, then the vanishing cycles $\{\alpha_1,\ldots,\alpha_{\delta}\}$ associated with the nodes are disjoint. 
\end{enumerate}
Consider $b\in B$ such that the curve $\mathcal{C}_b$ is nonsingular. As $\mathcal{C}_b$ is irreducible, the cohomology classes in $H^1(\mathcal{C}_b)$ of vanishing cycles are linearly independent, and can then be completed to a symplectic basis $\{\alpha_1,\beta_1,\ldots, \alpha_{\delta},\beta_{\delta}\}$.
Let $T_i$ be the generators of the (abelian) local fundamental group $\pi_1(B \setminus \Delta , b)\cong \mathbb{Z}^\delta$, where $T_i$ corresponds to “going around $D_i$”.\\ Setting $\tilde{B}:=B\setminus \Delta$ and denoting by $\tilde{\pi}:\tilde{C}\rightarrow \tilde{B}$ the corresponding restriction to the smooth locus of the family, the monodromy defining the local system $R^1\tilde{\pi}_*\mathbb{Q}$  on $\tilde{B}$ can be computed via the Picard-Lefschetz formula, and, in the symplectic basis above, the images of the generators of the fundamental group in $GL(H^1(\mathcal{C}_b))=GL(2\delta,\co)$ are given by block diagonal matrices consisting of one Jordan block of order 2 corresponding to a symplectic pair $\{\alpha_i,\beta_i\}$ and the identity elsewhere.
Also, as the vanishing cycles are independent, we can consider $R^1\tilde{\pi}_*\mathbb{Q}$ as direct sum of $\delta$ modules $V_i$ of rank 2 whose basis is $\{\alpha_i,\beta_i\}$. This makes much more easier to compute the invariants of any local system obtained by linear algebra operations from $R^1\tilde{\pi}_*\mathbb{Q}$. 
In our case we observe that, as $\mathcal{C}_b$ is nonsingular then
$$\mathcal{C}_b^{[m,m+1]}=\mathcal{C}_b^{(m,m+1)}=\mathcal{C}_b^{(m)}\times \mathcal{C}_b=\mathcal{C}_b^{[m]}\times \mathcal{C}_b.$$
By the MacDonald formula for the cohomology of the symmetric product we have 
\begin{equation}\label{mac}
R^i\tilde{\pi}_*^{[m]}\mathbb{Q}:=\bigoplus_{k=0}^{[\frac{i}{2}]}\bigwedge^{i-2k}R^1\tilde{\pi}_*\mathbb{Q}(-k)\cong R^{2m-i}\tilde{\pi}_*\mathbb{Q}(m-i)
\end{equation}
where $(-k)$ denotes the Tate weight shift of $(k,k)$ in the mixed Hodge structure on the cohomology.  
We define $\mathbb{S}^{i,m}$ the linear algebra operation on $R^1\tilde{\pi}_*\mathbb{Q}$  in the above formula, so that we have 
$$R^i\tilde{\pi}_*^{[m]}\mathbb{Q}=\mathbb{S}^{i,m}(R^1\tilde{\pi}_*\mathbb{Q}).$$
Applying the K\"{u}nneth formula and recalling that the cohomology of any curve $\mathcal{C}_b $ in the smooth locus has a pure Hodge structure given by 
$$R^0\tilde{\pi}_*\mathbb{Q}=\mathbb{Q} \quad  R^1\tilde{\pi}_*\mathbb{Q}\cong \mathbb{Q}^{2\delta} \quad R^2\tilde{\pi}_*\mathbb{Q}\cong \mathbb{Q}(-1)$$
we conclude that 
\begin{equation}\label{run}
R^i\tilde{\pi}_*^{[m,m+1]}\mathbb{Q}:= (R^i\tilde{\pi}^{[m]}\mathbb{Q} )\oplus (R^{i-1}\tilde{\pi}_*^{[m]}\mathbb{Q}\otimes R^1\tilde{\pi}_*\mathbb{Q}) \oplus (R^{i-2}\tilde{\pi}_*^{[m]}\mathbb{Q}(-1)).
\end{equation}
As before, let us call $\mathbb{T}^{i,m}$ the linear algebra operation on $R^1\tilde{\pi}_*\mathbb{Q})$ in the above formula.
%Call $\mathbb{T}^{i,m}$ the linear algebra operation we apply to on $R^1\tilde{\pi}_*\mathbb{Q}$ to obtain $R^1\tilde{\pi}_*^{[m,m+1]}\mathbb{Q}$: 
In view of equations \eqref{mac} and \eqref{run} may rewrite $\mathbb{T}^{i,m}$ as
$$R^i\tilde{\pi}_*^{[m,m+1]}\mathbb{Q}=\mathbb{T}^{i,m}(R^1\tilde{\pi}_*\mathbb{Q}):=\bigoplus_{j=0}^{2}  \mathbb{S}^{i-j,m}(R^1\tilde{\pi}_*\mathbb{Q})\otimes R^j\tilde{\pi}_*\mathbb{Q}. $$ 
Then there exists natural isomorphisms
$$\left(\mathbb{S}^{i,m}H^1(\mathcal{C}_b)\right)^{\pi_1(B \setminus \Delta)}\cong \mathcal{H}^0\left(IC_B(R^i\tilde{\pi}_*^{[m]}\mathbb{Q})\right)_{b_0}$$
$$\left(\mathbb{T}^{i,m}H^1(\mathcal{C}_b)\right)^{\pi_1(B \setminus \Delta)}\cong \mathcal{H}^0\left(IC_B(R^i\tilde{\pi}_*^{[m,m+1]}\mathbb{Q})\right)_{b_0}$$
between the monodromy invariants on $\mathbb{S}^{i,m}H^1(\mathcal{C}_b)$( resp $\mathbb{S}^{i,m}H^1(\mathcal{C}_b)$ ) and the stalk at $b_0$ of the first non-vanishing cohomology sheaf of the intersection cohomology complex of $R^i\tilde{\pi}_*^{[m]}\mathbb{Q}$ (resp. $R^i\tilde{\pi}_*^{[m,m+1]}\mathbb{Q}$).
The decomposition theorem implies that $H^*(\mathcal{C}_{b_0}^{[m]})$ and $H^*(\mathcal{C}_{b_0}^{[m,m+1]})$ contain respectively the Hodge structures
$$ \mathbb{H}^m:=\bigoplus_i \left(\mathbb{S}^{i,m}H^1(\mathcal{C}_b)\right)^{\pi_1(B \setminus \Delta)}$$
$$ \mathbb{I}^m:=\bigoplus_i \left(\mathbb{T}^{i,m}H^1(\mathcal{C}_b)\right)^{\pi_1(B \setminus \Delta)}$$
as a summand. We want to show that this is the unique summand by proving that the weight polynomial of the cohomology of the nested Hilbert scheme of the $\mathcal{C}_{b_0}$ is equal to the weight polynomial of $\mathbb{I}^m$. In that case the theorem will follow from proposition (\ref{critpesi}).

\begin{prop}\label{weights}
	Under the previous assumptions the following holds
	\begin{enumerate}[(i)]
		\item 	 $\mathfrak{w}(\mathcal{C}_{b_0}^{[m]})=\mathfrak{w}(\mathbb{H}^m)$
		\item $\mathfrak{w}(\mathcal{C}_{b_0}^{[m,m+1]})=\mathfrak{w}(\mathbb{I}^m)$
	\end{enumerate}
\end{prop}

\subsection{Hilbert scheme case}
Let $\pi:\mathcal{C}\rightarrow B$ a locally versal deformation of a singular rational nodal curve $\mathcal{C}_{b_0}=:C$. As a warm up for the nested case, we will compute the weight polynomial of $C^{[m]}$ and the weight polynomial of the Hodge structure $\mathbb{H}^{m}$ given by the monodromy invariants and show they are equal thus proving theorem \cite[Theorem 1]{MS1}.

\subsubsection{Computation of $\mathfrak{w}(C^{[m]}$)}

To compute $\mathfrak{w}(C^{[m]})$ we use power series to find a formula for the class of $C^{[m]}$ in the Grothendieck group.
First we notice that 
\begin{equation} \label{h1}
\sum_m q^m\left[C^{[m]}\right] =\sum_m q^m\left[C_{reg}^{[m]}\right]\prod_{x_i}\sum q^m\left[ C_{x_i}^{[m]}\right]
\end{equation}

As $C_{reg}=\mathbb{P}^1\setminus 2\delta \text{ regular points }p_1,\ldots p_{2\delta}$ then 
$$  \sum_m q^m\left[(\mathbb{P}^1)^{[m]}\right] =\sum_m q^m\left[C_{reg}^{[m]}\right]\prod_{p_i}\sum q^m\left[ C_{p_i}^{[m]}\right]$$
Now observe that  $(\mathbb{P}^1)^{[m]}=\mathbb{P}^{m}$; also as the $p_i$ are regular points  $\left[ C_{p_i}^{[m]}\right]=1$ for all $m$ and we have:
$$\dfrac{1}{(1-q)(1-q\mathbb{L})}=\sum_m q^m\left[C_{reg}^{[m]}\right]\dfrac{1}{(1-q)^{2\delta}}\Rightarrow \sum_m q^m\left[C_{reg}^{[m]}\right]=\dfrac{(1-q)^{2\delta-1}}{(1-q\mathbb{L})}$$
where $\mathbb{L}$ denotes the weight polynomial of the affine line.\\
Now, in \cite{R} Ran shows that $C_x^{[m]}$ consists of $m-1$ copies of $\mathbb{P}^1$ with $m-2$ intersections. Thus
$$\prod_{x_i}\sum q^m\left[ C_{x_i}^{[m]}\right]=\left(\sum q^m( (m-1)\mathbb{L}+1)\right)^{\delta}=\dfrac{\left(1-q+q^2\mathbb{L}\right)^{\delta}}{(1-q)^{2	\delta}}.$$
Substituting in equation (\ref{h1}), we get
$$ \sum_m q^m\left[C^{[m]}\right] =\dfrac{\left(1-q+q^2\mathbb{L}\right)^{\delta}}{(1-q)(1-q\mathbb{L})}$$

Expanding the series, we have that the coefficient of $q^m$ in the series is given by
\begin{equation}\mathfrak{w}(C^{[m]})=\sum_{s=0}^m(-1)^s\sum_{t=0}^{\delta}\binom{\delta}{t}\binom{t}{s-t}\mathbb{L}^{s-t}\cdot \sum_{l=0}^{m-s}\mathbb{L}^{l}=\sum_{s=0}^m(-1)^s\sum_{l=0}^{\delta}\binom{\delta}{t}\binom{t}{s-t}\mathbb{L}^{s-t}\cdot\dfrac{\mathbb{L}^{m-s+1}-1}{\mathbb{L}-1}.
\end{equation}

\subsubsection{Computation of $\mathfrak{w}(\mathbb{H}^m)$}
Let $b$ a point in the smooth locus. Since $$\mathbb{H}^m:=\bigoplus_i \left(\mathbb{S}^{i,m}H^1(\mathcal{C}_b)\right)^{\pi_1(B \setminus \Delta)},$$
to find $\mathfrak{w}(\mathbb{H}^m)$, by additivity of weight polynomials, we just need to compute the weight polynomial of each summand.

We start by computing the invariants in the cohomology groups $H^i(\mathcal{C}_b)$ of the monodromy $\rho: \pi_1(B\setminus \Delta)\rightarrow H^1(\mathcal{C}_b)$.  
Also, we recall that all the vanishing cycles $\alpha_i$ have weight 0, while $\beta_i$ have weight 2.\\  
In view of the definition of $\mathbb{S}^{i,m}H^1(\mathcal{C}_b)$, we need to understand the invariants of $\bigwedge^l H^1(\mathcal{C}_b)$ for any $l\geq0$. As we observed before, $H^1(\mathcal{C}_b)$ can be viewed as a direct sum of 2-dimensional representations $V_i$ on which a generator $T_j\in SL(2\delta,\co)$ of the monodromy acts as the identity if $i\neq j$ and $T_i(\alpha_i)=\alpha_i$, $T_i(\beta_i)=\alpha_i+\beta_i$. 
Thus $H^1(\mathcal{C}_b)=\bigoplus_{i=1}^{\delta} V_i$ and we have
\begin{equation}\label{hilbinv}
\bigwedge^l H^1(\mathcal{C}_b)=\bigoplus_{l_1+\ldots+l_{\delta}=l}\bigwedge^{l_1}V_1\otimes \ldots \otimes\bigwedge^{l_\delta}V_{\delta},\qquad 0\leq l_i\leq 2.
\end{equation}
Also, as $\dim V_i=2$
$$\bigwedge^{l_i}V_i=
\begin{cases}
\co \quad \qquad \mbox{ if }l_i=0\\
V_i \quad \qquad \mbox{ if }l_i=1\\
\co(-1) \ \quad \mbox{ if }l_i=2\\
\end{cases}.$$
The only invariants of $V_i$ are the $\alpha_i$, of weight 0. 
In conclusion we have that for any $i=0,\ldots, m$ we have
$$I(i,\delta):=\mathfrak{w}\left((\mathbb{S}^{i,m}H^1(\mathcal{C}_b))^{\pi_1(B\setminus\Delta)}\right)=
                      \mathfrak{w}\left((H^i(\mathcal{C}_b))^{\pi_1(B\setminus\Delta)}\right)=(-1)^i\sum_{k=0}^{[\frac{i}{2}]}\mathbb{L}^k\sum_{j=0}^{[\frac{i-2k}{2}]}\binom{\delta}{j}\binom{\delta-j}{i-2k-2j}\mathbb{L}^j$$
where the index $k$ is the one in MacDonald formula of $\mathbb{S}^{i,m}$ and $\mathbb{L}^k$ represent the weight shift $(-k)$. The second sum $j$ represents the number of second external power we take in (\ref{hilbinv}). 
The coefficient $\binom{\delta}{j}$ is the number of choices of $j$ second external powers among $\delta$ terms, while the coefficient $\binom{\delta-j}{i-2k-2j}$ is the number of choices of the remaining terms in (\ref{hilbinv}).\\ 

Summing over $m$ and taking the duality in  (\ref{mac}) into account we get
\begin{align}\label{hm}
\mathfrak{w}(\mathbb{H}^{m}) &= \sum_{i=0}^{m-1}(-1)^i(1+\mathbb{L}^{m-i})\sum_{k=0}^{[\frac{i}{2}]}\mathbb{L}^k\sum_{j=0}^{[\frac{i-2k}{2}]}\binom{\delta}{j}\binom{\delta-j}{i-2k-2j}\mathbb{L}^j\\
&+(-1)^{m}\sum_{k=0}^{[\frac{m}{2}]}\mathbb{L}^k\sum_{j=0}^{[\frac{m-2k}{2}]}\binom{\delta}{j}\binom{\delta-j}{m-2k-2j}\mathbb{L}^j
\end{align}

\begin{proof}[Proof of point (i) in Proposition \ref{weights}]
	We start looking at $\mathfrak{w}(\mathbb{H}^{m})$. First we notice that due to properties of binomial coefficient, the sum over $j$ goes to $\delta$ while the sum in $k$ can go to infinity. Also we have that $ \binom{\delta}{j}\binom{\delta-j}{i-2k-2j}=\binom{\delta}{i-2k-j}\binom{i-2k-j}{j}$.\\ Setting $l=i-2k-j$  and applying the remarks above we get
	\begin{align*}
	\mathfrak{w}(\mathbb{H}^{m}) &= \sum_{i=0}^{m-1}(-1)^i(1+\mathbb{L}^{m-i})\sum_{k=0}^{\infty}\mathbb{L}^k\sum_{l=0}^{\delta}\binom{\delta}{l}\binom{l}{i-2k-l}\mathbb{L}^{i-2k-l}+\\
	&+(-1)^{m}\sum_{k=0}^{\infty}\mathbb{L}^k\sum_{l=0}^{	\delta}\binom{\delta}{l}\binom{l}{m-2k-l}\mathbb{L}^{m-2k-l}
	\end{align*}
	Set $s=i-2k$ and split the sum in two parts with respect to the product with $(1+\mathbb{L}^{m-i})$.
	\begin{align*}
	\mathfrak{w}(\mathbb{H}^{m}) &= \sum_{s=0}^{m}(-1)^s\sum_{k=0}^{\infty}\mathbb{L}^k\sum_{l=0}^{\delta}\binom{\delta}{l}\binom{l}{s-l}\mathbb{L}^{s-l}+\\
	&+\sum_{s=0}^{m}(-1)^s\mathbb{L}^{m-s}\sum_{k=0}^{\infty}\mathbb{L}^{-k}\sum_{l=0}^{\delta}\binom{\delta}{l}\binom{l}{s-l}\mathbb{L}^{s-l}
	\end{align*} 
    Taking out the sums in $k$ and recalling that $\sum_{k=0}^{\infty}\mathbb{L}^k=\dfrac{1}{1-\mathbb{L}}$
	\begin{align*}
	\mathfrak{w}(\mathbb{H}^{m}) &= \dfrac{1}{1-\mathbb{L}}\sum_{s=0}^{m}(-1)^s\sum_{l=0}^{\delta}\binom{\delta}{l}\binom{l}{s-l}\mathbb{L}^{s-l}+\\
	&-\dfrac{\mathbb{L}}{1-\mathbb{L}}\sum_{s=0}^{m}(-1)^s\mathbb{L}^{m-s}\sum_{k=0}^{\infty}\mathbb{L}^{-k}\sum_{l=0}^{\delta}\binom{\delta}{l}\binom{l}{s-l}\mathbb{L}^{s-l}=\\
	&=\dfrac{1}{1-\mathbb{L}}\sum_{s=0}^{m}(-1)^s(1-\mathbb{L}^{m-s+1})\sum_{l=0}^{\delta}\binom{\delta}{l}\binom{l}{s-l}\mathbb{L}^{s-l}
	\end{align*}
	which is precisely $\mathfrak{w}(C^{[m]})$.
\end{proof}

%================
%NESTED CASE
%================
\subsection{Nested Hilbert scheme case}

As above suppose $\pi:\mathcal{C}\rightarrow B$ is a locally versal deformation of a singular rational nodal curve $\mathcal{C}_{b_0}=:C$.  We now want to show point $(ii)$ of proposition (\ref{weights}), to conclude the proof of theorem \ref{teo}.
Again, we compute the weight polynomials $\mathfrak{w}(C^{[m,m+1]})$,  $\mathfrak{w}(\mathbb{I}^{m})$ and show that their are equal. 

\subsubsection{Computation of $\mathfrak{w}(C^{[m,m+1]})$}

We start by stratifying $\mathcal{C}_{b_0}^{[m,m+1]}$. As the weight polynomial depends only on the class in the Grothendieck group, we can work there.
Let $\mathcal{C}_{b_0,reg}:=\mathcal{C}_{b_0}\setminus\{x_1,\ldots,x_{\delta}\}$. We can consider the colength 1 ideal of $\cohilbnest$ as a copy of $\mathcal{C}_{b_0}^{[m]}$ to which we add a further point $p\in\cohilb$. Whenever we add a regular point $p$ the class does not change, while when the point is a node we need to be careful about the number of occurrences of the node in the colength one ideal. In \cite{R}, Ran shows that the nested Hilbert scheme $C_x^{[k,k+1]}$ supported on one node, consists of $2k-1$ copies of $\mathbb{P}^1$ with $2k-2$ intersections. \\
As a consequence $\left[C_x^{[k,k+1]} \right] =(2k-1)\mathbb{L}+1$. \\

We stratify $\mathcal{C}_{b_0}^{[m,m+1]}$ with respect to the number of times the nodes appear in $\left[\mathcal{C}_{b_0}^{[m]}\right]$:
\begin{align*}
\left[ \mathcal{C}_{b_0}^{[m,m+1]}\right] &= \left[\mathcal{C}_{b_0}^{[m]}\times \mathcal{C}_{b_0,reg}\right]+ \sum_{i=1}^{\delta} \sum_{k=0}^m \left[(\mathcal{C}_{b_0}- x_i)^{[m-k]}\times C_{x_i}^{[k,k+1]} \right]= \\
&= \left[\mathcal{C}_{b_0}^{[m]}\times \mathcal{C}_{b_0,reg}\right]+ \delta \sum_{k=0}^m \left[(\mathcal{C}_{b_0}- x)^{[m-k]}\times C_{x}^{[k,k+1]} \right]
\end{align*}
We observe that for any $k\geq 0$ we can write $\left[C_x^{[k,k+1]}\right]=\left[C_x^{[k]}\right]+k\mathbb{L}$. Making a substitution in the above equation we get
$$  \left[ \mathcal{C}_{b_0}^{[m,m+1]}\right] = \left[\mathcal{C}_{b_0}^{[m]}\times \mathcal{C}_{b_0,reg})\right]+ \delta \sum_{k=0}^m \left[(\mathcal{C}_{b_0}- x)^{[m-k]}\times C_{x}^{[k]} \right]+\delta \mathbb{L}\sum_{k=0}^m k\left[(\mathcal{C}_{b_0}- x)^{[m-k]}\right].$$
Since $\sum_{k=0}^m \left[(\mathcal{C}_{b_0}- x)^{[m-k]}\times C_{x}^{[k]}\right]=\left[\cohilb\right]$, we have that the second term of the sum consists precisely of those $\delta$ copies of $\cohilb$ which, added to the first term, give $\cohilb\times \mathcal{C}_{b_0}$.
Finally, we notice that $(\mathcal{C}_{b_0}-x)$ can be considered as a curve $\tilde{\mathcal{C}}$ with $\delta-1$ nodes minus two regular points $p,q$. Then the class of its Hilbert scheme can be computed as $\left[\tilde{\mathcal{C}}^{[m]}\right]=\sum_{k=0}^m \left[(\mathcal{C}_{b_0}- x)^{[m-k]}\right]\times C_{p,q}^{[k]}$, where $C_{p,q}^{[k]}$ is the Hilbert scheme with support $p\cup q$. As $p$ and $q$ are regular points, $\left[C_{p,q}^{[k]}\right]$ is just the number of length non ordered $k-$ple in $p, q$ , which is equal to $k$.\\
In conclusion we can write
\begin{equation}\label{strathilb}
\left[ \mathcal{C}_{b_0}^{[m,m+1]}\right] =\cohilb\times \mathcal{C}_{b_0}+\delta \mathbb{L}\left[\tilde{\mathcal{C}}^{[m]}\right]
\end{equation}

\subsubsection{Computation of $\mathfrak{w}(\mathbb{I}^{m}$)}

We remind that
$$H^i(\mathcal{C}_b^{[m,m+1]})=H^{i}(\mathcal{C}_b^{[m]})\oplus H^{i-1}(\mathcal{C}_b^{[m]})\otimes H^1(\mathcal{C}_b)\oplus H^{i-2}(\mathcal{C}_b^{[m]})(-1).$$
We notice that, by applying the MacDonald formula to second term we get
$$H^{i-1}(\mathcal{C}_b^{[m]})\otimes H^1(\mathcal{C}_b)=\bigoplus_{k=0}^{[\frac{i-1}{2}]}\bigwedge^{i-1-2k}H^1(\mathcal{C}_b)\otimes H^1(\mathcal{C}_b) (-k)$$
 As a result we will have to find both the invariants of $\bigwedge^l H^1(\mathcal{C}_b)$ and those of $\bigwedge^l H^1(\mathcal{C}_b)\otimes H^1(\mathcal{C}_b)$.
 We have seen how to find the invariants of $\bigwedge^l H^1(\mathcal{C}_b)$ in the computation for the Hilbert scheme; when looking at the invariants of $\bigwedge^l H^1(\mathcal{C}_b)\otimes H^1(\mathcal{C}_b)$ we have to be more careful: there is more than just the invariant of $\bigwedge^l H^1(\mathcal{C}_b)$ times the invariant of $H^1(\mathcal{C}_b)$.\\
 Let us be more precise: recall that $H^1(\mathcal{C}_b)=\bigoplus_{i=1}^{\delta} V_i$ and that we have
 \begin{equation}
 \bigwedge^l H^1(\mathcal{C}_b)=\bigoplus_{l_1+\ldots+l_{\delta}=l}\bigwedge^{l_1}V_1\otimes \ldots \otimes\bigwedge^{l_\delta}V_{\delta},\qquad 0\leq l_i\leq 2.
 \end{equation}
 Also, as $\dim V_i=2$
 $$\bigwedge^{l_i}V_i=
 \begin{cases}
 \co \quad \qquad \mbox{ if }l_i=0\\
 V_i \quad \qquad \mbox{ if }l_i=1\\
 \co(-1) \ \quad \mbox{ if }l_i=2\\
 \end{cases}.$$
 Thus 
\begin{equation}\label{invr1}
\bigwedge^l H^1(\mathcal{C}_b)\otimes H^1(\mathcal{C}_b)=(\bigoplus_{l_1+\ldots+l_{\delta}=l}\bigwedge^{l_1}V_1\otimes \ldots \otimes\bigwedge^{l_\delta}V_{\delta})\otimes(V_1\oplus\ldots \oplus V_{\delta}).
\end{equation}

By the considerations above, the monodromy invariants of summands of type $V_i\otimes V_j$ for $i\neq j$ are just an invariant of $V_i$ tensor an invariant of $V_j$, while invariants of summands of type $\bigwedge^2 V_i \otimes V_j$ are just the invariants of $V_i$ with shifted weight. \\
The invariants which are not the tensor product of an invariant of $\bigwedge^l H^1(\mathcal{C}_b)$ times an invariant of $H^1(\mathcal{C}_b)$ come from the summands $V_i\otimes V_i=\bigwedge^2 V_i\otimes Sym^2(V_i)$. These summands provide additional invariants of weight 2, which are those of $\bigwedge^2 V_i$.

As equation (\ref{invr1}) is symmetric in the $V_i$'s it is sufficient to compute the invariants of $$(\bigoplus_{l_1+\ldots+l_{\delta}=l}\bigwedge^{l_1}V_1\otimes \ldots \otimes\bigwedge^{l_\delta}V_{\delta})\otimes V_1 $$
and multiply what we obtain by $\delta$.\\
If $l_1\neq 1$ then the formula we wrote for the Hilbert scheme still holds, while when $l_1=1$ we have a certain number of invariants of weight 2 to take into account.
\begin{align*}
\mathfrak{w}\left((H^i(\mathcal{C}_b^{[m]})\otimes H^1(\mathcal{C}_b))^{\pi_1(B\setminus\Delta)}\right) &=\delta\sum_{k=0}^{[\frac{i}{2}]}\mathbb{L}^k\sum_{j=0}^{[\frac{i-2k}{2}]}\binom{\delta-1}{j-1}\binom{\delta-j}{i-2k-2j}\mathbb{L}^j+\\
&+(1+\mathbb{L})\binom{\delta-1}{j}\binom{\delta-1-j}{i-2k-2j-1}\mathbb{L}^j+\\
&+\binom{\delta-1}{j}\binom{\delta-1-j}{i-2k-2j}\mathbb{L}^j
\end{align*}
The first term in the sum represents the case in which $l_1=2$, the second one is the case of $l_1=1$ and the last one is $l_1=0$. As in the previous formula, the index $k$ is the one in the MacDonald formula, while the index $j$ represents the number of $l_i\neq l_1$ that are equal to 2.  \\
Summing over $i$ we get 
\begin{align*}
\mathfrak{w}(\mathbb{I}^m)&=\sum_{i=0}^m (1+\mathbb{L}^{m+1-i}) \sum_{k=0}^{[\frac{i}{2}]}\mathbb{L}^k\sum_{j=0}^{[\frac{i-2k}{2}]}\binom{\delta}{j}\binom{\delta-j}{i-2k-2j}\mathbb{L}^j+\\ &+\delta\sum_{k=0}^{[\frac{i-1}{2}]}\mathbb{L}^k\sum_{j=0}^{[\frac{i-1-2k}{2}]}\binom{\delta-1}{j-1}\binom{\delta-j}{i-1-2k-2j}\mathbb{L}^j+(1+\mathbb{L})\binom{\delta-1}{j}\binom{\delta-1-j}{i-1-2k-2j-1}\mathbb{L}^j+\\
&+\binom{\delta-1}{j}\binom{\delta-1-j}{i-1-2k-2j}\mathbb{L}^j+\mathbb{L}\sum_{k=0}^{[\frac{i-2}{2}]}\mathbb{L}^k\sum_{j=0}^{[\frac{i-2-2k}{2}]}\binom{\delta}{j}\binom{\delta-j}{i-2-2k-2j}\mathbb{L}^j+\\
&+(-1)^{m+1}\delta\sum_{k=0}^{[\frac{m}{2}]}\mathbb{L}^k\sum_{j=0}^{[\frac{m-2k}{2}]}\binom{d-1}{j-1}\binom{d-j}{m-2k-2j}\mathbb{L}^j+\\
&+(1+\mathbb{L})\binom{\delta-1}{j}\binom{\delta-1-j}{m-2k-2j-1}\mathbb{L}^j+\binom{\delta-1}{j}\binom{\delta-1-j}{m-2k-2j}\mathbb{L}^j+\\
&+2\mathbb{L}\sum_{k=0}^{[\frac{m-1}{2}]}\mathbb{L}^k\sum_{j=0}^{[\frac{m-1-2k}{2}]}\binom{\delta}{j}\binom{\delta-j}{m-1-2k-2j}\mathbb{L}^j.
\end{align*}

Looking at equation $(\ref{strathilb})$ we want to separate the invariants which are the tensor product of invariants of the Hilbert scheme and the invariants of the curve from those coming from the weight 2 part of the pieces $V_i\otimes V_i$, which we will prove to be precisely the invariants of the Hilbert scheme of the curves with $\delta-1$ nodes. 
We want to show that the former are 
\begin{align*}
A=&\sum_{i=0}^m (1+\mathbb{L}^{m+1-i}) \sum_{k=0}^{[\frac{i}{2}]}\mathbb{L}^k\sum_{j=0}^{[\frac{i-2k}{2}]}\binom{\delta}{j}\binom{\delta-j}{i-2k-2j}\mathbb{L}^j+\\ &+\delta\sum_{k=0}^{[\frac{i-1}{2}]}\mathbb{L}^k\sum_{j=0}^{[\frac{i-1-2k}{2}]}\binom{\delta-1}{j-1}\binom{\delta-j}{i-1-2k-2j}\mathbb{L}^j+\binom{\delta-1}{j}\binom{\delta-1-j}{i-2k-2j-1}\mathbb{L}^j+\\
&+\binom{\delta-1}{j}\binom{\delta-1-j}{i-1-2k-2j}\mathbb{L}^j+\mathbb{L}\sum_{k=0}^{[\frac{i-2}{2}]}\mathbb{L}^k\sum_{j=0}^{[\frac{i-2-2k}{2}]}\binom{\delta}{j}\binom{\delta-j}{i-2-2k-2j}\mathbb{L}^j+\\
&+(-1)^{m+1} 2\mathbb{L}\sum_{k=0}^{[\frac{m-1}{2}]}\mathbb{L}^k\sum_{j=0}^{[\frac{m-1-2k}{2}]}\binom{\delta}{j}\binom{\delta-j}{m-1-2k-2j}\mathbb{L}^j+\delta\sum_{k=0}^{[\frac{m}{2}]}\mathbb{L}^k\sum_{j=0}^{[\frac{m-2k}{2}]}\binom{\delta-1}{j-1}\binom{\delta-j}{m-2k-2j}\mathbb{L}^j+\\
&+\binom{\delta-1}{j}\binom{\delta-1-j}{m-1-2k-2j}\mathbb{L}^j+\binom{\delta-1}{j}\binom{\delta-1-j}{i-1-2k-2j}\mathbb{L}^j+\\
&+2\mathbb{L}\sum_{k=0}^{[\frac{m-1}{2}]}\mathbb{L}^k\sum_{j=0}^{[\frac{m-1-2k}{2}]}\binom{\delta}{j}\binom{\delta-j}{m-1-2k-2j}\mathbb{L}^j.
\end{align*}
while the latter are
\begin{align*}
B&=\delta \mathbb{L} \sum_{i=0}^m(-1)^i(1+\mathbb{L}^{m+1-i}) \sum_{k=0}^{[\frac{i-1}{2}]}\mathbb{L}^k\sum_{j=0}^{[\frac{i-1-2k}{2}]}\binom{\delta-1}{j}\binom{\delta-1-j}{i-2k-2j-1}\mathbb{L}^j+\binom{\delta-1}{j}\binom{\delta-1-j}{i-1-2k-2j}\mathbb{L}^j+\\
&(-1)^{m+1}\delta\mathbb{L}\sum_{k=0}^{[\frac{m}{2}]}\mathbb{L}^k\sum_{j=0}^{[\frac{m-2k}{2}]}+\binom{\delta-1}{j}\binom{\delta-1-j}{m-2k-2j-1}\mathbb{L}^j.
\end{align*}

\begin{lem}
	$$ A=\left[\cohilb\times \mathcal{C}_{b_0}\right]=\mathfrak{w}(\mathbb{H}^m)(\mathbb{L}-\delta+1)$$
\end{lem}
\begin{proof}
	First we notice that, due to properties of binomial coefficients, the quantity
	\begin{align*}
	&\sum_{k=0}^{[\frac{i-1}{2}]}\mathbb{L}^k\sum_{j=0}^{[\frac{i-1-2k}{2}]}\binom{\delta-1}{j-1}\binom{\delta-j}{i-1-2k-2j}\mathbb{L}^j+\binom{\delta-1}{j}\binom{\delta-1-j}{i-2k-2j-1}\mathbb{L}^j+\binom{\delta-1}{j}\binom{\delta-1-j}{i-1-2k-2j}\mathbb{L}^j
	\end{align*}
	is equal to 
	$$ \sum_{k=0}^{[\frac{i-1}{2}]}\mathbb{L}^k\sum_{j=0}^{[\frac{i-1-2k}{2}]}\binom{\delta}{j}\binom{\delta-j}{i-1-2k-2j}\mathbb{L}^j= I(i-1,\delta);$$
	thus
	$$A=\sum_{i=0}^m(-1)^i(1+\mathbb{L}^{m+1-i}) I(i,\delta)+\delta I(i-1,\delta)+\mathbb{L}I(i-2,\delta)+(-1)^{m+1}\left( 2\mathbb{L}I(m-1,\delta)+\delta I(m,\delta)\right).$$ 
	Now, up t a change of indexes $i'=i-1$, we see that 
	$$\sum_{i=0}^m(-1)^i(1+\mathbb{L}^{m+1-i})\delta I(i-1,\delta)+\delta(-1)^{m+1}I(m,\delta)=-\delta \mathfrak{w}(\mathbb{H}^{[m]}).$$
	Also,
	\begin{align*}
	&\sum_{i=0}^m(-1)^i(1+\mathbb{L}^{m+1-i}) I(i,\delta)+\mathbb{L}I(i-2,\delta)+(-1)^{m+1} 2\mathbb{L}I(m-1,\delta)=\\
	&\sum_{i=0}^m(-1)^i(1+\mathbb{L}^{m+1-i}) I(i,\delta)+\sum_{i=0}^m(-1)^i(\mathbb{L}+\mathbb{L}^{m+2-i}) +(-1)^{m+1} 2\mathbb{L}I(m-1,\delta).
	\end{align*}
	If one again sets $i'=i-2$, this becomes
	\begin{align*}
	&(1+\mathbb{L})\sum_{i=0}^{m-2}(-1)^i\mathbb{L}^{m-i})I(i,\delta)+\sum_{i=0}^{m-2}(-1)^i\mathbb{L}^{m-i} I(i,\delta)+\\
	&+(-1)^{m+1} 2\mathbb{L}I(m-1,\delta)+(-1)^m I(m,d)(1+\mathbb{L})+(-1)^{m-1}(1+\mathbb{L}^2)=\\
	&=(1+\mathbb{L})\left(\sum_{i=0}^{m-1}(-1)^i\mathbb{L}^{m-i})I(i,\delta)+(-1)^{m}I(m,d)\right)=(1+\mathbb{L})\mathfrak{w}(\mathbb{H}^m).
	\end{align*}
\end{proof}
Analogously, using properties of binomial coefficients and setting $i'=i-1$ one can prove
\begin{lem}
	$$B=\delta\mathbb{L}\left[\tilde{\mathcal{C}}^{[m]}\right]$$
\end{lem}

and this complete the proof of proposition (\ref{weights}) and Theorem 1.\\

\subsection*{Acknowledgements} I wish to thank my supervisor Luca Migliorini for suggesting me this problem as long as for the countless comments and corrections. Then I would like to thank Filippo Viviani and Gabriele Mondello for the helpful comments and suggestions. I also thank the anonymous reviewers for pointing out a mistake in the previous version of this paper as long as several imprecisions and misprints. Finally I am grateful to Enrico Fatighenti, Danilo Lewanski, Giovanni Mongardi and Marco Trozzo for the support in writing this article. 
%BIBLIOGRAFIA

\vspace{40pt}
\noindent
Dr. Camilla Felisetti\\
Department of Mathematics,\\
University of Bologna,\\
Piazza di Porta San Donato 5\\
40126 Bologna\\
Italy\\
camilla.felisetti2@unibo.it\\
		\end{document}